\documentclass[a4paper]{amsart}

\usepackage[english]{babel}
\usepackage[utf8]{inputenc}
\usepackage[T1]{fontenc}
\usepackage{lmodern}

\usepackage{natbib}

\usepackage{amssymb} 
\usepackage{version} 

\def\EuScript{\mathcal}
\def\mathscr{\EuScript}

\newcommand{\defegal}{\triangleq}                                   

\newcommand{\bbR}{\mathbb{R}}                               

\newcommand{\Bp}[1]{\Big(#1\Big)}                           

\newcommand{\fcara}[1]{\chi_{#1}}                           
\newcommand{\findi}[1]{\mathbf{1}_{#1}}                     

\newcommand{\tribu}[1]{\mathscr{#1}}                        
\newcommand{\omeg}{\Omega}                                  
\newcommand{\trib}{\tribu{A}}                               
\newcommand{\prbt}{\mathbb{P}}                              
\newcommand{\espe}{\mathbb{E}}                              

\newcommand{\va}[1]{\boldsymbol{\uppercase{#1}}}            

\newcommand{\Bigdelim}[1]{\Bp{#1}}                          
\newcommand{\vardelim}[1]{\left(#1\right)}                  

\newcommand{\Besp}[2][]{\espe_{#1}\Bigdelim{#2}}            

\newcommand{\proba}[2][]{\prbt_{#1}\vardelim{#2}}           
\newcommand{\esper}[2][]{\espe_{#1}\vardelim{#2}}           

\theoremstyle{plain}

\newtheorem{lemma}{Lemma}
\theoremstyle{definition}

\theoremstyle{remark}
\newtheorem{remark}{Remark}
\newtheorem{hyp}{Assumption}
\newtheorem{example}{Example}

\usepackage[pdftex,hypertexnames=false,colorlinks=true,linkcolor=black,citecolor=black]{hyperref}

\begin{document}

\title{Dynamic consistency for Stochastic Optimal Control problems}

\author[P. Carpentier]{Pierre Carpentier}
\address{P. Carpentier, ENSTA ParisTech,
         32, boulevard Victor, 75739 Paris Cedex 15, Fran\-ce.}
\email{pierre.carpentier@ensta.fr}

\author[J.-Ph. Chancelier]{Jean-Philippe Chancelier}
\address{J.-Ph. Chancelier, Universit\'{e} Paris-Est, CERMICS, \'{E}cole des Ponts ParisTech,
         6 \& 8 avenue Blaise Pascal, 77455 Marne-la-Vall\'{e}e Cedex 2.}
\email{jpc@cermics.enpc.fr}

\author[G. Cohen]{Guy Cohen}
\address{G. Cohen, Universit\'{e} Paris-Est, CERMICS, \'{E}cole des Ponts ParisTech,
         6 \& 8 avenue Blaise Pascal, 77455 Marne-la-Vall\'{e}e Cedex 2.}
\email{guy.cohen@mail.enpc.fr}

\author[M. De Lara]{Michel De Lara}
\address{M. De Lara, Universit\'{e} Paris-Est, CERMICS, \'{E}cole des Ponts ParisTech,
         6 \& 8 avenue Blaise Pascal, 77455 Marne-la-Vall\'{e}e Cedex 2.}
\email{delara@cermics.enpc.fr}

\author[P. Girardeau]{Pierre Girardeau}
\address{P. Girardeau, EDF R\&D, 1, avenue du G\'{e}n\'{e}ral de Gaulle,
         F-92141 Clamart Cedex, France,
         also with Universit\'{e} Paris-Est, CERMICS and ENSTA.}
\email{pierre.girardeau@cermics.enpc.fr}

\thanks{This study was made within the Systems and Optimization Working
Group (SOWG), which is composed of Laetitia Andrieu, Kengy Barty,
Pierre Carpentier, Jean-Philippe Chancelier, Guy Cohen, Anes
Dallagi, Michel De Lara and Pierre Girardeau, and based at
Universit\'{e} Paris-Est, CERMICS, Champs sur Marne, 77455 Marne la
Vall\'{e}e Cedex 2, France.}

\date{\today}


\begin{abstract}
For a sequence of dynamic optimization problems, we aim at
discussing a notion of consistency over time. This notion can be
informally introduced as follows. At the very first time
step~$t_0$, the decision maker formulates an optimization problem
that yields optimal decision rules for all the forthcoming time
step~$t_0, t_1, \dots, T$; at the next time step~$t_1$, he is able
to formulate a new optimization problem starting at time~$t_1$
that yields a new sequence of optimal decision rules. This process
can be continued until final time~$T$ is reached. A family of
optimization problems formulated in this way is said to be time
consistent if the optimal strategies obtained when solving the
original problem remain optimal for all subsequent problems. The
notion of time consistency, well-known in the field of Economics,
has been recently introduced in the context of risk measures,
notably by~\citet{Artzner_AOR_2007} and studied in the Stochastic
Programming framework by~\citet{Shapiro_ORL_2009} and for Markov
Decision Processes (MDP) by~\citet{Ruszczynski_OO_2009}. We here
link this notion with the concept of ``state variable'' in MDP,
and show that a significant class of dynamic optimization problems
are dynamically consistent, provided that an adequate state
variable is chosen.
\end{abstract}

\maketitle

\section{Introduction}

Stochastic Optimal Control (SOC) is concerned with sequential
decision-making under uncertainty. Consider a dynamical process
that can be influenced by exogenous noises as well as decisions
one has to make at every time step. The decision maker wants to
optimize the behavior of the dynamical system (for instance,
minimize a production cost) over a certain time horizon. As the
system evolves, observations of the system are made; we here
suppose that the decision maker is able to keep in memory all the
past observations. Naturally, it is generally more profitable for
him to adapt its decisions to the observations he makes of the
system. He is hence looking for strategies rather than simple
decisions. In other words, he is looking for applications that map
every possible history of the observations to corresponding
decisions. Because the number of time steps may be large, the
representation of such an object is in general numerically
intractable.

However, an amount of information lighter than the whole history
of the system is often sufficient to make an optimal decision. In
the seminal work of~\citet{Bellman57}, the minimal information on
the system that is necessary to make the optimal decision plays a
crucial role; it is called the \emph{state
variable}~\citep[see][for a more formal definition]{Whittle}.
Moreover, the Dynamic Programming (DP) principle provides a way to
compute the optimal strategies when the state space dimension is
not too large~\citep[see][for a broad overview on
DP]{BertsekasDP}. The aim of this paper is to establish a link
between the concept of state variable and the notion of time
consistency\footnote{We either use the term ``dynamically
consistent'' or ``time consistent'' to refer to the same notion.}.

The notion of dynamic consistency is well-known in the field of
economics~\citep[see][]{Hammond:1989} and has been introduced in
the context of risk measures \citep[see][for definitions and
properties of coherent and consistent dynamic risk
measures]{Artzner_AOR_2007,Riedel_SPA_2004,Detlefsen_FS_2005,Cheridito_EJP_2006}.
Dynamic consistency has then been studied in the stochastic
programming framework by \citet{Shapiro_ORL_2009} and for Markov
Decision Processes by \citet{Ruszczynski_OO_2009}. In this paper,
we rather use the (almost equivalent) definition of time
consistency given by \cite{Ekeland_arXiv_2006}, which is more
intuitive and seems better suited in the framework of optimal
control problems. In this context, the property of time
consistency is loosely stated as follows. The decision maker
formulates an optimization problem at time~$t_0$ that yields a
sequence of optimal decision rules for~$t_0$ and for the following
time steps~$t_1, \dots, t_N=T$. Then, at the next time step~$t_1$,
he formulates a new problem starting at~$t_1$ that yields a new
sequence of optimal decision rules from time steps~$t_1$ to~$T$.
Suppose the process continues until time~$T$ is reached. The
sequence of optimization problems is said to be dynamically
consistent if the optimal strategies obtained when solving the
original problem at time~$t_{0}$ remain optimal for all subsequent
problems. In other words, time consistency means that strategies
obtained by solving the problem at the very first stage do not
have to be questioned later on.

The notion of information here plays a crucial role. Indeed, we
show in this paper that a sequence of problems may be consistent
for some information structure while inconsistent for a different
one. Consider for example a standard stochastic optimization
problem solvable using DP. We will observe that the sequence of
problems formulated after the original one at the later time steps
are time consistent. Add now a probabilistic constraint involving
the state at the final time~$T$. We will show that such a
constraint brings time inconsistency in the sense that optimal
strategies based on the usual state variable have to be
reconsidered at each time step. This is because, roughly speaking,
a probabilistic constraint involves not only the state variable
values but their probabilistic distributions. Hence the only
knowledge of the usual state variable of the system is
insufficient to write consistent problems at subsequent time
steps. So, in addition to the usual technical difficulties
regarding probabilistic constraints (mainly related to the
non-convexity of the feasible set of strategies), an additional
problem arises in the dynamic case. We will see that, in fact,
this new matter comes from the information on which the optimal
decision is based. Therefore, with a well-suited state variable,
the sequence of problems regains dynamic consistency.

In~\S\ref{sec:Det}, we carefully examine the notion of time
consistency in the context of a deterministic optimal control
problem. The main ideas of the paper are so explained and then
extended, in~\S\ref{sec:Sto}, to a sequence of SOC problems. Next,
in~\S\ref{sec:Constrained}, we show that simply adding a
probability constraint (or, equivalently in our context, an
expectation constraint) to the problem makes time consistency fall
apart, when using the original state variable. We then establish
that time consistency can be recovered provided an adequate state
variable is chosen. We conclude that, for a broad class of SOC
problems, time consistency has to be considered with respect to
the notion of a state variable and of DP.

\section{A first example} \label{sec:Det}

We introduce sequential deterministic optimal control problems,
indexed by time, and derive the notion of time consistency on this
instance. We then illustrate the fact that the decision making
process may be time consistent or not, depending on the
information on which decisions are based. The discussion is
informal, in the sense that we do not enter technical details
regarding existence of the solutions for the problems we
introduce.

Let us consider a discrete and finite time
horizon~$t_0,\dots,t_N=T$.\footnote{where~$t_i+1=t_{i+1}$} The
decision maker has to optimize (according to a cost function we
introduce below) the management of an amount of stock~$x_t$, which
lies in some space~$\mathcal{X}_t$, at every time step~$t=t_0,
\dots, T$. Let~$\mathcal{U}_t$ be some other space, for every time
step~$t=t_0, \dots, T-1$. At each time step~$t$, a decision~$u_t
\in \mathcal{U}_t$ has to be made. Then a cost~$L_t$ is incurred
by the system, depending on the values of the control and on the
auxiliary variable~$x_t$ that we call the state of the system.
This state variable is driven from time~$t$ to time~$t+1$ by some
dynamics~$f_t: \mathcal{X}_t \times \mathcal{U}_t \rightarrow
\mathcal{X}_{t+1}$. The aim of the decision maker is to minimize
the sum of the intermediate costs~$L_t$ at all time steps plus a
final cost~$K$.

The problem hence reads:
\begin{subequations} \label{eqn:P}
\begin{align}
\min_{x, u} \quad & \sum_{t=t_0}^{T-1} L_t\left(x_t, u_t\right)
    + K\left(x_T\right), \\
\intertext{subject to the initial condition:}
    & x_{t_0} \text{ given}, \\
\intertext{and dynamic constraints:}
    & x_{t+1} = f_t\left(x_t, u_t\right), \qquad \forall t=t_0, \dots, T-1.
\end{align}
\end{subequations}
Note that here the decision at time~$t$ is taken knowing the
current time step and the initial condition (the decision is
generally termed ``open loop''). A priori, there is no need for
more information since the model is deterministic.

Suppose a solution to this problem exists. This is a sequence of
controls that we denote by~$u_{t_0,t_0}^*, \dots, u_{t_0,T-1}^*$,
where the first index refers to the initial time step and the
second index refers to the time step for which the decision
applies. Moreover, we suppose a solution exists for each one of
the natural subsequent problems, i.e. for every~$t_i=t_1, \dots,
T-1$:
\begin{subequations} \label{eqn:Pti}
\begin{align}
\min_{x, u} \quad& \sum_{t=t_i}^{T-1} L_t\left(x_t, u_t\right) + K\left(x_T\right), \\
\text{s.t.} \quad& x_{t_i} \text{ given}, \\
& x_{t+1} = f_t\left(x_t, u_t\right), \qquad \forall t=t_i, \dots,
T-1.
\end{align}
\end{subequations}
We denote the solutions of these problems
by~$u_{t_i,t_i}^*,\dots,u_{t_i,T-1}^*$, for every time step~$t_i=t_1, \dots,
T-1$. Those notations however make implicit the fact that the
solutions do generally depend on the initial condition~$x_{t_i}$.
We now make a first observation.

\begin{lemma}[Independence of the initial condition] \label{ppty:Indpdt}
In the very particular case when the solution to
Problem~\eqref{eqn:P} and the solutions to
Problems~\eqref{eqn:Pti} for every time step~$t_i=t_1, \dots, T-1$
do not depend on the initial state conditions, problems are
dynamically consistent.
\end{lemma}

\begin{proof}
Let us denote by~$x_{t_0, t_i}^*$ the optimal value of the state
variable within Problem~\eqref{eqn:P} at time~$t_i$. If we suppose
that solutions to Problems~\eqref{eqn:Pti} do not depend on the
initial condition, then they are the same as  the solutions
obtained with the initial condition~$x_{t_0, t_i}^*$, namely
$u_{t_0,t_i}^*, \dots, u_{t_0,T-1}^*$. In other words, the
sequence of decisions~$u_{t_0,t_0}^*$, $\dots$, $u_{t_0,T-1}^*$
remains optimal for the subsequent problems starting at a later
date.
\end{proof}

This property is of course not true in general, but we see in
Example~\ref{ex:Det} hereafter and in~\S\ref{sec:Sto} that some
very practical problems do have this surprising property.

\begin{example} \label{ex:Det}
Let us introduce, for every~$t=t_0, \dots, T-1$, functions~$l_t:
\mathcal{U}_t \rightarrow \mathbb{R}$ and~$f_t: \mathcal{U}_t
\rightarrow \mathbb{R}$, and assume that~$x_t$ is scalar. Let~$K$
be a scalar constant and consider the following deterministic
optimal control problem:
\begin{align*}
\min_{x, u} \quad& \sum_{t=t_0}^{T-1} l_t\left(u_t\right) x_t  +
K x_T, \\
\text{s.t.} \quad& x_{t_0} \text{ given}, \\
& x_{t+1} = f_t\left(u_t\right) x_t, \qquad \forall t=t_0, \dots,
T-1.
\end{align*}
Variables $x_t$ can be recursively replaced using dynamics~$f_t$.
Therefore, the above optimization problem can be written:
\begin{equation*}
\min_{u} \; \sum_{t=t_0}^{T-1} l_t\left(u_t\right)
f_{t-1}\left(u_{t-1}\right) \ldots f_{t_0}\left(u_{t_0}\right)
x_{t_0} + K f_{T-1}\left(u_{T-1}\right) \ldots f_{t_0}
\left(u_{t_0}\right) x_{t_0}.
\end{equation*}
Hence the optimal cost of the problem is linear with respect to
the initial condition~$x_{t_0}$. Suppose that~$x_{t_0}$ only takes
positive values. Then the value of~$x_{t_0}$ has no influence on
the minimizer (it only influences the optimal cost). The same
argument applies at subsequent time steps~$t_i>t_0$ provided that
dynamics are such that~$x_t$ remains positive for every time
step~$t=t_1, \dots, T$. Now, formulate the same problem at a later
date~$t_i=t_1, \dots, T-1$, with initial condition~$x_{t_i}$
given. By the same token as for the first stage problem, the value
of the initial condition~$x_{t_i}$ has no influence on the optimal
controls. Assumptions made in Lemma~\ref{ppty:Indpdt} are
fulfilled, so that the time consistency property holds true for
open-loop decisions without reference to initial state conditions.

Although, for the time being, this example may look very special,
we will see later on that it is analogous to familiar SOC
problems.
\end{example}

As already noticed, Lemma~\ref{ppty:Indpdt} is not true in
general. Moreover, the deterministic formulation~\eqref{eqn:P}
comes in general from the representation of a real-life process
which may indeed be subject to unmodelized disturbances. Think of
an industrial context, for example, in which sequential decisions
are taken in the following manner.
\begin{itemize}
    \item At time~$t_0$, Problem~\eqref{eqn:P} is solved. One obtains
        a decision~$u_{t_0, t_0}^*$ to apply at time~$t_0$, as well as
        decisions~$u_{t_0, t_1}^*$, $\dots$, $u_{t_0, T-1}^*$ for future
        time steps.
    \item At time~$t_1$, one formulates and solves the
        problem starting at time~$t_1$ with initial
        condition~$x_{t_1}=f_{t_0}(x_{t_0}, u_{t_0, t_0}^*) +
        \varepsilon_{t_1}$, $\varepsilon_{t_1}$ being some perturbation of
        the model. There is
        no reason not to use the observation of the actual value of the
        variable~$x_{t_1}$ at time~$t_1$ as long as we have it at our disposal.
    \item Hence a decision~$u_{t_1, t_1}^*$ is obtained, which is
        different from the initially obtained optimal decision~$u_{t_0,
        t_1}^*$ (once again, in general).
    \item The same process continues at times~$t_2,\ldots, T-1$.
\end{itemize}
Let us now state the two following lemmas.

\begin{lemma}[True deterministic world] \label{ppty:Det}
    If the deterministic model is actually exact, i.e. if all
    perturbations~$\varepsilon_{t_i}$ introduced above equal zero,
    then Problems~\eqref{eqn:Pti} with initial conditions~$x_{t_i} =
    x_{t_i}^* \defegal f_{t_i}(x_{t_{i-1}}^*, u_{t_0, t_{i-1}}^*)$ are
    dynamically consistent.
\end{lemma}

\begin{proof}
    Since decisions~$u_{t_0, t_0}^*, \dots, u_{t_0, T-1}^*$ are
    optimal for Problem~\eqref{eqn:P}, it follows that
    decisions $u_{t_0, t_1}^*, \dots, u_{t_0, T-1}^*$
    are optimal for the problem:
    \begin{align*}
        \min_{x, u} \quad&  L_t\left(x_{t_0}, u_{t_0, t_0}^*\right) +
            \sum_{t=t_1}^{T-1} L_t\left(x_t, u_t\right) + K\left(x_T\right), \\
        \text{s.t.} \quad& x_{t_1} = f_{t_1}(x_{t_0}, u_{t_0, t_0}^*), \\
        & x_{t+1} = f_t\left(x_t, u_t\right), \qquad \forall t=t_1, \dots,
        T-1,
    \end{align*}
    which has the same arg min as Problem~\eqref{eqn:Pti} at time~$t_1$. The
    same argument applies recursively for subsequent time steps.
\end{proof}

It is clear that Lemma~\ref{ppty:Det} is not satisfied in real
life. Therefore, adding disturbances to the problem seems to bring
inconsistency to the sequence of optimization problems. Decisions
that are optimal for the first stage problem do not remain optimal
for the subsequent problems if we do not let decisions depend on
the initial conditions.

In fact, as it is stated next, time consistency is recovered
provided we let decisions depend upon the right information.

\begin{lemma}[Right amount of information] \label{ppty:RightInfo}
Suppose that one is looking for strategies~$(\Phi_{t_0,t_0}^*$,
$\dots$, $\Phi_{t_0, T-1}^*)$ as feedback functions depending on
the variable~$x$. Then Problems~\eqref{eqn:Pti} are time
consistent for every time step~$t=t_0, \dots, T-1$.
\end{lemma}

\begin{proof}
    The result is a direct application of the DP principle, which states that
    there exists such a feedback
    function~$\Phi^*_{t_0, t_i}$
    that is optimal for Problem~\eqref{eqn:P} and is still optimal for
    Problem~\eqref{eqn:Pti} at time~$t_i$, whatever initial condition~$x_{t_i}$
    is.
\end{proof}

We thus retrieve the dynamic consistency property provided that we
use the feedback functions $\Phi_{t_0,t}^*$ rather than the
controls~$u_{t_{0},t}^*$. In other words, problems are dynamically
consistent as soon as the control strategy is based on a
sufficiently rich amount of information (time instant~$t$
\textbf{and} state variable~$x$ in the deterministic case).

There is of course an obvious link between these optimal
strategies and the controls
$(u_{t_{0},t_{0}}^*,\ldots,u_{t_{0},T-1}^*)$, namely:
\begin{align*}
u_{t_0,t}^* &= \Phi_{t_0,t}^*\left(x_{t_0,t}^*\right), \qquad
\forall t=t_0, \dots, T-1, \intertext{where}
x_{t_0,t_0}^* &= x_{t_0}, \\
x_{t_0,t+1}^* &= f_t\left(x_{t_0,t}^*,
\Phi_{t_0,t}^*\left(x_{t_0,t}^*\right)\right),
\qquad \forall t=t_0, \dots, T-1. \\
\end{align*}

The considerations we made so far seem to be somewhat trivial
However, we whall observe that for SOC problems, which may seem
more complicated at first sight, the same considerations remain
true. Most of the time, decision making processes are time
consistent, provided we choose the correct information on which
decisions are based.

\section{Stochastic optimal control without constraints} \label{sec:Sto}

We now consider a more general case in which a controlled
dynamical system is influenced by modeled exogenous disturbances.
The decision maker has to find strategies to drive the system so
as to minimize some objective function over a certain time
horizon. This is a sequential decision making process on which we
can state the question of dynamic consistency. As in the previous
example, the family of optimization problems is derived from the
original one by truncating the dynamics and the cost function (the
final time step~$T$ remains unchanged in each problem), and
strategies are defined relying on the same information structure
as in the original problem. In the sequel, random variables will
be denoted using bold letters.

\subsection{The classical case} \label{ssec:Classical}

Consider a dynamical system characterized by state\footnote{The
use of the terminology ``state'' is somewhat abusive until we make
Assumption~\ref{hyp:Markov}.} variables
$\va{x}=(\va{x}_t)_{t=t_0,\dots,T}$, where~$\va{x}_t$ takes values
in~$\mathcal{X}_t$. The system can be influenced by control
variables $\va{u}=(\va{u}_t)_{t=t_0,\dots,T-1}$ and by exogenous
noise variables $\va{w} = (\va{w}_t)_{t=t_0, \dots, T}$
($\va{u}_t$ and~$\va{w}_t$ taking values in~$\mathcal{U}_t$
and~$\mathcal{W}_t$ respectively). All random variables are
defined on a probability space~$(\omeg, \trib, \prbt)$. The
problem we consider consists in minimizing the expectation of a
sum of costs depending on the state, the control and the noise
variables over a discrete finite time horizon. The state variable
evolves with respect to some dynamics that depend on the current
state, noise and control values. The problem starting at~$t_0$
writes:\footnote{We here use the notations~$\sim$ for ``is
distributed according to'' and~$\preceq$ for ``is measurable with
respect to''.}
\begin{subequations} \label{eqn:Psto}
\begin{align}
\min_{\va{x}, \va{u}} \quad& \esper{\sum_{t=t_0}^{T-1}
L_t\left(\va{x}_t, \va{u}_t, \va{w}_{t+1}\right) +
K\left(\va{x}_T\right)}, \\
\text{s.t.} \quad& \va{x}_{t_0} \text{ given}, \\
& \va{x}_{t+1} = f_t\left(\va{x}_t, \va{u}_t, \va{w}_{t+1}\right),
\qquad \forall t=t_0, \dots, T-1, \\
& \va{u}_t \preceq \va{x}_{t_0}, \va{w}_{t_1}, \dots, \va{w}_t,
\qquad \forall t = t_0, \dots, T-1.
\end{align}
\end{subequations}
Noises that affect the system can be correlated through time. A
general approach in optimal control consists in including all
necessary information in the variable~$\va{x}$ so that
variables~$\va{w}_{t_1}, \dots, \va{w}_T$ are independent through
time. At most, one has to include all the past values of the noise
variable within the variable~$\va{x}$. We hence make the following
assumption.
\begin{hyp}[Markovian setting] \label{hyp:Markov}
    Noises variables~$\va{x}_{t_0}, \va{w}_{t_1}, \dots, \va{w}_T$ are
    independent.
\end{hyp}
Using Assumption~\ref{hyp:Markov}, it is well
known~\citep[see][]{BertsekasDP} that:
\begin{itemize}
    \item there is no loss of optimality in looking for the optimal
        strategy~$\va{u}_t$ at time~$t$ as a feedback function depending
        on the state variable~$\va{x}_t$, i.e. as a (measurable) function of
        the form~$\Phi_{t_0, t}: \mathcal{X}_t \rightarrow \mathcal{U}_t$;
    \item the optimal strategies~$\Phi_{t_0, t_0}^*, \dots, \Phi_{t_0,
        T-1}^*$ can be obtained by solving the classical DP equation. Let
        $V_t(x)$ denote the optimal cost when being at time step~$t$ with state
        value~$x$, this equation reads:
        \begin{align*}
            V_{T}(x) & = K(x), \\
            V_{t}(x) & = \min_{u} \Besp{L_{t}(x,u,\va{w}_{t+1}) + V_{t+1}\big(
            f_{t}(x,u,\va{w}_{t+1}) \big)}.
        \end{align*}
\end{itemize}
We call this case the \emph{classical case}. It is clear while
inspecting the DP equation that optimal strategies~$\Phi_{t_0,
t_0}^*$, $\dots$, $\Phi_{t_0, T-1}^*$ remain optimal for the
subsequent optimization problems:
\begin{subequations} \label{eqn:PbDyn}
\begin{align}
\min_{\va{x}, \va{u}} \quad& \esper{\sum_{t=t_i}^{T-1}
L_t\left(\va{x}_t, \va{u}_t, \va{w}_{t+1}\right) +
K\left(\va{x}_T\right)}, \\
\text{s.t.} \quad& \va{x}_{t_i} \text{ given}, \\
& \va{x}_{t+1} = f_t\left(\va{x}_t, \va{u}_t, \va{w}_{t+1}\right),
\qquad \forall t=t_i, \dots, T-1, \\
& \va{u}_t \preceq \va{x}_{t_i}, \va{w}_{t_{i+1}}, \dots,
\va{w}_t, \qquad \forall t = t_i, \dots, T-1,
\end{align}
\end{subequations}
for every~$t_i=t_1, \dots, T-1$. In other words, these problems
are dynamically consistent provided the information variable at
time $t$ contains at least the state variable~$\va{x}_{t}$. While
building an analogy with properties described in the deterministic
example in~\S\ref{sec:Det}, the reader should be aware that the
case we consider here is closer to Lemma~\ref{ppty:Indpdt} than to
Lemma~\ref{ppty:RightInfo}, as we explain now in more details.

\subsection{The distributed formulation}

Another consequence of the previous DP equation for
Problem~\eqref{eqn:Psto} is that the optimal feedback functions do
not depend on the initial condition~$\va{x}_{t_0}$. The
probability law of~$\va{x}_{t_0}$ only affects the optimal cost
value, but not its arg min. In fact, we are within the same
framework as in Example~\ref{ex:Det}. Indeed,
Problem~\eqref{eqn:Psto} can be written as a deterministic
distributed optimal control problem involving the probability laws
of the state variable, the dynamics of which are given by the
so-called Fokker-Planck equation. Let us detail this last
formulation~\citep[see][]{Witsenhausen:1973}.

Let~$\Psi_t$ be the space of~$\bbR$-valued functions
on~$\mathcal{X}_t$. Denoting~$\mu_{t_0}$ the probability law of
the first stage state~$\va{x}_{t_0}$, and given feedback
laws~$\Phi_t: \mathcal{X}_t \rightarrow \mathcal{U}_t$ for every
time step~$t=t_0, \dots, T-1$, we define the
operator~$A_t^{\Phi_t}: \Psi_{t+1} \rightarrow \Psi_t$, which is
meant to integrate cost functions backwards in time,
as\footnote{We do not aim at discussing technical details
concerning integrability here. We suppose that operators we
introduce are well-defined.}:
\begin{equation*}
\left(A_t^{\Phi_t} \psi_{t+1}\right)\left(\cdot\right) \defegal
\esper{\psi_{t+1} \circ f_t\left(\cdot, \Phi_t\left(\cdot\right),
\va{w}_{t+1}\right)}.
\end{equation*}
Given a feedback function~$\Phi_t$ and a cost function~$\psi_{t+1}
\in \Psi_{t+1}$,~for every~$x \in \mathcal{X}_t$ the
value~$(A_t^{\Phi_t} \psi_{t+1})(x)$ is the expected value
of~$\psi_{t+1}(\va{x}_{t+1})$, knowing that~$\va{x}_t=x$ and that
feedback~$\Phi_t$ is used. Thanks to a duality argument, the
Fokker-Planck equation, which describes the evolution of the state
probability law (as driven by the chosen feedback laws~$\Phi_t$),
is obtained:
\begin{equation*}
    \mu_{t+1} = \left(A_t^{\Phi_t}\right)^\star \mu_t,
\end{equation*}
with~$(A_t^{\Phi_t})^\star$ being the adjoint operator
of~$A_t^{\Phi_t}$. Next we introduce the operator
$\Lambda_t^{\Phi_t} : \mathcal{X}_t \rightarrow \mathbb{R}$:
\begin{equation*}
\Lambda_t^{\Phi_t}\left(\cdot\right) \defegal
\esper{L_t\left(\cdot, \Phi_t\left(\cdot\right),
\va{w}_{t+1}\right)},
\end{equation*}
which is meant to be the expected cost at time~$t$ for each
possible state value when feedback function~$\Phi_t$ is applied.
Let us define, for every~$\psi_t \in \Psi_t$ and every probability
law~$\mu_t$ on~$\mathcal{X}_t$,~$\langle \psi_t, \mu_t \rangle$
as~$\esper{\psi_t(\va{x}_t)}$ when~$\va{x}_t$ is distributed
according to~$\mu_t$. We can now write a deterministic
infinite-dimensional optimal control problem that is equivalent to
Problem~\eqref{eqn:Psto}:
\begin{align*}
\min_{\Phi, \mu} \quad& \sum_{t=t_0}^{T-1} \left\langle
\Lambda_t^{\Phi_t}, \mu_t \right\rangle +
\left\langle K, \mu_T \right\rangle, \\
\text{s.t.} \quad& \mu_{t_0} \text{ given}, \\
& \mu_{t+1} = \left(A_t^{\Phi_t}\right)^\star \mu_t, \qquad
\forall t=t_0, \dots, T-1.
\end{align*}
\begin{remark}
An alternative formulation is:
\begin{align*}
\min_{\Phi, \psi} \quad&
\left\langle \psi_{t_{0}}, \mu_{t_{0}} \right\rangle, \\
\text{s.t.} \quad& \psi_{T}=K, \\
& \psi_{t} = A_t^{\Phi_t}\psi_{t+1}+\Lambda_t^{\Phi_t}, \qquad
\forall t=T-1,\dots,t_0 .
\end{align*}
This may be called \emph{``the backward formulation''} since the
``state'' \(\psi_{t}(\cdot)\) follows an affine dynamics which is
backward in time, with an initial-only cost function (whereas the
previous forward formulation follows a forward linear dynamics
with an integral \(+\) final cost function). Both formulations are
infinite-dimensional linear programming problems which are dual of
each other. The functions \(\mu(\cdot)\) and \(\psi(\cdot)\) are
the distributed state and/or co-state (according to which one is
considered the primal problem) of this distributed deterministic
optimal control problem of which \(\Phi\) is the distributed
control.
\end{remark}

Probability laws~$\mu_t$ are by definition positive and appear
only in a multiplicative manner in the problem. Hence we are in a
similar case as Example~\ref{ex:Det}. The main difference is
rather technical: since we here have probability laws instead of
scalars, we need to apply backwards in time interversion theorems
between expectation and minimization in order to prove that the
solution of the problem actually does not depend on the initial
condition~$\mu_{t_0}$. Indeed, suppose that~$\mu_{T-1}$ is given
at time step~$T-1$. Then the most inner optimization problem
reads:
\begin{align*}
    \min_{\Phi_{T-1}} \quad& \left\langle \Lambda_t^{\Phi_{T-1}}, \mu_{T-1}
    \right\rangle + \left\langle K, \mu_T \right\rangle, \\
    \text{s.t.} \quad& \mu_T = \left(A_{T-1}^{\Phi_{T-1}}\right)^\star \mu_{T-1},
\intertext{which is equivalent to:}
    \min_{\Phi_{T-1}} \quad& \left\langle \Lambda_t^{\Phi_{T-1}}
    + A_{T-1}^{\Phi_{T-1}} K, \mu_{T-1} \right\rangle.
\end{align*}
The point is that operators~$\Lambda_t^{\Phi_{T-1}}+
A_{T-1}^{\Phi_{T-1}} K$ and~$\mu_{T-1}$ both take values
in~$\mathcal{X}_{T-1}$ and that the minimization has to be done
``$x$ by~$x$'', so that we are in the case of Example~\ref{ex:Det}
for every~$x$. Therefore, the minimizer does not depend
on~$\mu_{T-1}$. For a rigorous proof, one needs several technical
assumptions concerning measurability, which we do not intend to
discuss in this paper~\citep[see][Theorem 14.60]{RockWets}. The
same argument applies recursively to every time step before~$T-1$
so that, at time~$t_0$, the initial condition~$\mu_{t_0}$ only
influences the optimal cost of the problem, but not the argument
of the minimum itself (here, the feedback laws~$\Phi_{t_0,t}^*$).

Hence, following Lemma~\ref{ppty:Indpdt},
Problems~\eqref{eqn:PbDyn} are naturally time consistent when
strategies are searched as feedback functions on~$\va{x}_t$ only.
It thus appears that the rather general class of stochastic
optimal control problems shaped as Problem~\eqref{eqn:Psto} is in
fact very specific. However, such a property does not remain true
when adding new ingredients in the problem, as we show in the next
subsection.

\section{Stochastic optimal control with constraints}\label{sec:Constrained}

We now give an example in which the state variable, as defined
notably by~\citet{Whittle}, cannot be reduced to
variable~$\va{x}_t$ as above. Let us make Problem~\eqref{eqn:Psto}
more complex by adding to the model a probability constraint
applying to the final time step $T$. For instance, we want the
system to be in a certain state at the final time step with a
given probability:
\begin{equation*}
    \proba{h\left(\va{x}_T\right) \geq b} \leq
    \pi .
\end{equation*}
Such chance constraints can equivalently be modelled as an
expectation constraint in the following way:
\begin{equation*}
    \esper{\findi{\left\{h\left(\va{x}_T\right) \geq b\right\}}} \leq \pi,
\end{equation*}
where~$\findi{A}$ refers to the indicator function of set~$A$.
Note however that chance constraints bring important theoretical
and numerical difficulties, notably regarding connexity and
convexity of the feasible set of controls, even in the static
case. The interested reader should refer to the work
of~\citet{PrekopaStoProg}, and to the handbook
by~\citet[Ch.5]{StochasticProgramming03} for mathematical
properties and numerical algorithms in Probabilistic
Programming~\citep[see also][for related
studies]{Henrion02,HenrionStrugarek06}. We do not discuss them
here. The difficulty we are interested in is common to both chance
and expectation constraints. This is why we concentrate in the
sequel on adding an expectation constraint to
Problem~\eqref{eqn:Psto} of the form:
\begin{equation*}
    \esper{g\left(\va{x}_T\right)} \leq a. \label{eqn:Constraint}
\end{equation*}
The reader familiar with chance constraints might want to see the
level~$a$ as a level of probability that one wants to satisfy for
a certain event at the final time step.

We now show that when adding such an expectation constraint, the
dynamic consistency property falls apart. More precisely, the
sequence of SOC problems are not time consistent anymore when
using the usual state variable. Nevertheless, we observe that the
lack of consistency comes from an inappropriate choice for the
state variable. By choosing the appropriate state variable, one
regains dynamic consistency.

\subsection{Problem setting} \label{ssec:ProblemSetting}

We now go back to the constrained formulation and introduce a
measurable function~$g: \mathcal{X}_T \rightarrow \bbR$ and~$a \in
\bbR$. We consider Problem~\eqref{eqn:Psto} with the additional
final expectation constraint:
\begin{equation*}
    \esper{g\left(\va{x}_T\right)} \leq a.
\end{equation*}
The subsequent optimization problems formulated at an initial
time~$t_{i}>t_{0}$ are naturally deduced from this problem. The
level~$a$ of the expectation constraint remains the same for every
problem. One has to be aware that this corresponds to a (naive)
modelling choice for the family of optimization problems under
consideration. Such a choice is questionable since the perception
of the constraint may evolve over time.

Suppose there exists a solution for the problem at~$t_{0}$. As
previously, we are looking for the optimal control at time~$t$ as
a feedback function~$\Phi_{t_0, t}^*$ depending on the
variable~$\va{x}_{t}$. The first index~$t_{0}$ refers to the time
step at which the problem is stated, while the second index~$t$
refers to the time step at which the decision is taken.

One has to be aware that these solutions now implicitly depend on
the initial condition~$\va{x}_{t_{0}}$. Indeed, let~$\mu_T$ be the
probability law of~$\va{x}_T$. Constraint~\eqref{eqn:Constraint}
can be written~$\left\langle g, \mu_T \right\rangle \leq a$, so
that the equivalent distributed formulation of the initial time
problem is:
\begin{align*}
\min_{\Phi, \mu} \quad
    & \sum_{t=t_0}^{T-1}
      \left\langle \Lambda_t^{\Phi_t}, \mu_t \right\rangle +
      \left\langle K, \mu_T \right\rangle, \\
\intertext{subject to the Fokker-Planck dynamics:}
    & \mu_{t+1} = \left(A_t^{\Phi_t}\right)^\star \mu_t,
      \qquad \forall t=t_0, \dots, T-1,\\
\intertext{$\mu_{t_0}$ being given by the initial condition, and
the final expectation constraint:}
    & \left\langle g, \mu_T \right\rangle \leq a.
\end{align*}
Even though this problem seems linear with respect to variables
$\mu_{t}$, the last constraint introduces an additional highly
nonlinear term in the cost function, namely:
\begin{equation*}
\fcara{\left\{\left\langle g,\mu_T \right\rangle \leq a \right\}},
\end{equation*}
where~$\fcara{A}$ stands for the characteristic
function\footnote{as defined in convex analysis: $\fcara{A}(x) =
\Big\{\begin{array}{ll} 0 & \text{if } x \in A \\ +\infty &
\text{otherwise} \end{array}$} of set~$A$. The dynamics are still
linear and variables~$\mu_t$ are still positive, but the objective
function is not linear with respect to~$\mu_T$ anymore, and
therefore not linear with respect to the initial law~$\mu_{t_0}$
either. Hence there is no reason for feedback laws to be
independent of the initial condition as in the case without
constraint presented in~\S\ref{sec:Sto}.

Let us now make a remark on this initial condition. Since the
information structure is such that the state variable is fully
observed, the initial condition is in fact of a deterministic
nature:
\begin{equation*}
\va{x}_{t_{0}} = x_{t_{0}},
\end{equation*}
where~$x_{t_{0}}$ is a given (observed) value of the system state.
The probability law of~$\va{x}_{t_{0}}$ is accordingly the Dirac
function~$\delta_{x_{t_{0}}}$.\footnote{The initial
law~$\mu_{t_{0}}$ in Problem~\eqref{eqn:Psto} corresponds to the
information available on~$\va{x}_{t_{0}}$ \emph{before}
$\va{x}_{t_{0}}$ is observed, but it seems more reasonable in a
practical situation to use all the available information when
setting the problem again at each new initial time, and thus to
use a Dirac function as the initial condition.} The reasoning made
for the problem initiated at time~$t_{0}$ remains true for the
subsequent problems starting at time~$t_{i}$: an
observation~$x_{t_{i}}$ of the state variable~$\va{x}_{t_{i}}$
becomes available before solving Problem~\eqref{eqn:PbDyn}, so
that its natural initial condition is in fact:
\begin{equation*}
\va{x}_{t_{i}} = x_{t_{i}}.
\end{equation*}
Otherwise stated, the initial state probability law in each
optimization problem we consider should correspond to a Dirac
function. Note that such a sequence of Dirac functions is not
driven by the Fokker-Planck equation, but is in fact associated to
some dynamics of the degenerate filter corresponding to this
perfect observation scheme. In the sequel, we assume such an
initial condition for every problem we consider.

Now, according to Lemma~\ref{ppty:Det}, the subsequent
optimization problems formulated at time~$t_{i}$ will be
dynamically consistent provided their initial conditions are given
by the optimal Fokker-Planck equation:
\begin{equation*}
\mu_{t_0, t_{i}}^* =
\Big(A_{t_{i-1}}^{\Phi_{t_{0},t_{i-1}}^*}\Big)^\star \ldots
\Big(A_{t_{0}}^{\Phi_{t_{0},t_{0}}^*}\Big)^\star \mu_{t_0}.
\end{equation*}
However, except for noise free problems, such a probability
law~$\mu_{t_0, t_{i}}^*$ is always different from a Dirac
function, which is, as already explained, the natural initial
condition for the subsequent problem starting at time~$t_{i}$. As
a conclusion, the sequence of problems is not time consistent as
long as we consider feedback laws~$\Phi_{t}$ depending
on~$\va{x}_{t}$ only.

\begin{remark}[Joint probability constraints] \label{rem:PrbtConstraints}
Rather than~$\proba{g\left(\va{x}_T\right) \geq b} \leq a$, let us
consider a more general chance constraint of the form:
\begin{equation*}
\proba{g_t\left(\va{x}_t\right) \geq b_t, \forall t=t_1, \dots, T}
\leq a.
\end{equation*}
This last constraint can be modelled, like the previous one,
through an expectation constraint by introducing a new binary
state variable:
\begin{align*}
\va{y}_{t_0} &= 1, \\
\va{y}_{t+1} &= \va{y}_t \times
\findi{\left\{g_{t+1}\left(\va{x}_{t+1}\right) \geq
b_{t+1}\right\}}, \qquad \forall t=t_0, \dots, T-1,
\end{align*}
and considering constraint~$\esper{\va{y}_T} \leq a$.
\hfill$\square$
\end{remark}

\subsection{Back to time consistency}

We now show that time consistency can be recovered provided we
choose the right state variable on which to base decisions. We
hence establish a link between time consistency of a family of
optimization problems and the notion of state variable.

We claim that a better-suited state variable for the family of
problems with final time expectation constraint introduced above
is the probability law of the variable~$\va{x}$. Let us denote
by~$V_t(\mu_t)$ the optimal cost of the problem starting at
time~$t$ with initial condition~$\mu_t$. Using notations of the
distributed formulation of a SOC problem, one can write a DP
equation depending on the probability laws~$\mu$ on~$\mathcal{X}$:
\begin{align*}
V_T\left(\mu\right)
    & = \left\{ \begin{array}{ll}
           \left\langle K, \mu \right\rangle &
           \text{ if $\left\langle g, \mu \right\rangle \leq a$}, \\
           +\infty & \text{ otherwise},
        \end{array} \right. \\
\intertext{and, for every~$t=t_0, \dots, T-1$ and every
probability law~$\mu$ on~$\mathcal{X}$:} V_t\left(\mu\right)
    &=\min_{\Phi_t} \left\langle \Lambda_t^{\Phi_t}, \mu \right\rangle +
     V_{t+1}\left(\left(A_t^{\Phi_t}\right)^\star \mu\right).
\end{align*}
The context is similar to the one of the deterministic example
of~\S\ref{sec:Det}, and Lemma~\ref{ppty:RightInfo} states that
solving the deterministic infinite-dimensional problem associated
with the constrained problem leads to time consistency provided DP
is used. For the problem under consideration, we thus obtain
optimal feedback functions~$\Phi_{t}$ which depend on the
probability laws $\mu_{t}$. Otherwise stated, the family of
constrained problems introduced in~S\ref{ssec:ProblemSetting} is
time consistent provided one looks for strategies as feedback
functions depending on both the variable~$\va{x}_{t}$ and the
probability law of~$\va{x}_{t}$.

Naturally, this DP equation is rather conceptual. The resolution
of such an equation is intractable in practice since probability
laws~$\mu_t$ are infinite-dimensional objects.

\section{Conclusion}

We informally introduced a notion of time consistency of a
sequence of decision-making problems, which basically requires
that plans that are made from the very first time remain optimal
if one rewrites optimization problems at subsequent time steps. We
show that, for several classes of optimal control problems, this
concept is not new and can be directly linked with the notion of
state variable, which is the minimal information one must use to
be able to take the optimal decision.

We show that, in general, feedback laws have to depend on the
probability law of the usual state variable for Stochastic Optimal
Control problems to be time consistent. This is necessary, for
example, when the model contains expectation or chance
constraints.

Future works will focus on three main directions. The first
concern will be to better formalize the state notion in the vein
of the works by~\citet{Witsenhausen:1971a,Witsenhausen:1973}
and~\citet{Whittle}. The second will be to establish the link with
the literature concerning risk measures, in particular the work
by~\citet{Ruszczynski_OO_2009}. Finally, the last DP equations we
introduced are in general intractable. In a forthcoming paper, we
will provide a way to get back to a finite-dimensional information
variable, which makes a resolution by DP tractable.

\bibliographystyle{spmpsci}

\end{document}